\documentclass[11pt]{article}

\usepackage{amsmath}
\usepackage{amsthm}
\usepackage{amssymb}
\usepackage{amsfonts}
\usepackage{graphicx}
\usepackage{verbatim}
\usepackage{mathrsfs}
\usepackage{subfigure}
\usepackage{fancyhdr}
\usepackage{latexsym}
\usepackage{graphicx}
\usepackage{dsfont}
\usepackage{color}
\usepackage[hidelinks]{hyperref}
\usepackage{bm}
\usepackage{enumitem}

\theoremstyle{plain}
\newtheorem{theorem}{Theorem}[section]
\newtheorem{proposition}[theorem]{Proposition}
\newtheorem{lemma}[theorem]{Lemma}

\newtheorem{remark}[theorem]{Remark}

\newcommand{\R}{\mathbb{R}}
\newcommand{\Rd}{\mathbb{R}_*^d}
\newcommand{\N}{\mathbb{N}}
\newcommand{\Sphere}{\mathbb{S}}
\newcommand{\CC}{\mathcal{C}}

\newcommand{\bx}{\bar v}
\newcommand{\by}{\bar v_*}
\newcommand{\ba}{\bar a}
\newcommand{\pa}{\partial}

\DeclareMathOperator{\tr}{tr}
\DeclareMathOperator{\diverg}{div}
\DeclareMathOperator{\rot}{curl}
\DeclareMathOperator{\supp}{supp}
\DeclareMathOperator{\rank}{rank}
\DeclareMathOperator{\dist}{dist}
\DeclareMathOperator{\argsinh}{argsinh}

\title{Rigorous study of the equilibria of  collision kernels appearing in the theory of  weak turbulence}
\author{M. BREDEN$^a$ and L. DESVILLETTES$^b$ \\
\it $^a$ Technical University of Munich, \\ 
\it Faculty of Mathematics, Research Unit ``Multiscale and Stochastic 
Dynamics", \\
\it 85748 Garching b. M\"unchen, Germany, \\
and\\
\it $^b$ Universit\'{e} Paris Diderot, Sorbonne Paris Cit\'{e},\\
 \it Institut de Math\'{e}matiques de Jussieu-Paris Rive Gauche, UMR 7586, CNRS,\\ \it Sorbonne Universit\'{e}s, UPMC Univ. Paris 06, F-75013, Paris, France.} 
\date{ }

\begin{document}

\maketitle

\begin{abstract} In this paper, we rigorously obtain all the equilibria of collision kernels of type ``two particles give two particles'' appearing in weak turbulence theory under very general assumptions, thus completing the ``equality case'' in Boltzmann's H-theorem for those models. We also provide some rigorous results for collision kernels of type ``two particles give one particle'', under assumptions which include some of the most classical kernels of this type. The method of proof is inspired by the quantitative estimates obtained for the Landau equation in \cite{landau_jfa}.
\end{abstract}

\bigskip 

\bigskip 

\noindent
{\bf Keys words:}  weak turbulence, collision kernels, Boltzmann's H-theorem, equilibria\\

 
\section{Introduction }\label{intro}

Rarefied gases are described by their density $f(t,x,v)\ge 0$ in phase space, which satisfies the Boltzmann equation (cf. \cite{cercignani})
\begin{equation}
\label{eq:Boltzmann}
\pa_t f + v \cdot \nabla_x f = Q(f),
\end{equation}
where $t$ is the time, $x$ the position, and $v$ the velocity of a given particle. The operator $Q$, acting only on the velocity variable of $f$, is defined as 
 \begin{equation}
 Q(f)(v) = \int_{\R^3}\int_{\Sphere^2} \bigg( f(\frac{v+v_*}2 + \frac{|v-v_*|}2\,\sigma)\, f(\frac{v+v_*}2 - \frac{|v-v_*|}2\,\sigma)
 \end{equation}
 $$ - f(v) f(v_*) \bigg) \, B(|v-v_*|, \frac{v-v_*}{|v-v_*|}\cdot\sigma )\, d\sigma dv_*, $$
where $B$ is related to the cross section of the binary collision process between the molecules of the gas. 
\medskip

Boltzmann's H-theorem states that (for all functions $f$ for which the quantity makes sense) $\int Q(f) \, \ln f \,dv \le 0$, and that 
\begin{equation}
 (\forall v , \, Q(f)(v) = 0 ) \qquad \iff \qquad  \int Q(f) \, \ln f \,dv = 0 \qquad \iff \qquad \ln f \in {\hbox{ Vect }} (1, v_i, |v|^2).
\end{equation}
In other terms, the equilibria of the Boltzmann equation are the Maxwellian functions of $v$. 
\medskip 

In order to make the above statement completely rigorous, one needs to assume that $B>0$ (or at least that $B>0$ on some suitable set), and that $f$ is regular enough for the quantities to make sense (at least a.e.). 
\medskip

For many cross sections, it is in fact possible to transform this statement in an inequality between the entropy $\int f\,\ln f dv$ and the entropy dissipation $ \int Q(f) \, \ln f \,dv $. We refer to \cite{survey} and  
the references therein for results and conjectures in this direction.
\medskip

Boltzmann's H-theorem is based on the computation
$$\int  Q(f)(v)\, \ln f(v)\, dv  = -\frac14\int_{\R^3} \int_{\R^3}\int_{\Sphere^2} \bigg( f(\frac{v+v_*}2 + \frac{|v-v_*|}2\,\sigma)\, f(\frac{v+v_*}2 - \frac{|v-v_*|}2\,\sigma) - f(v) f(v_*) \bigg) $$
\begin{equation}
\,\times \, \bigg( \ln( f(\frac{v+v_*}2 + \frac{|v-v_*|}2\,\sigma)\, f(\frac{v+v_*}2 - \frac{|v-v_*|}2\,\sigma) ) 
- \ln ( f(v)\,  f(v_*) ) \bigg)\, B(|v-v_*|, \frac{v-v_*}{|v-v_*|}\cdot\sigma )\, d\sigma dv_* dv,
\end{equation}
so that (provided that $B>0$), it amounts to say that if for all $v,v_* \in \R^3$, $\sigma \in \Sphere^2$,
\begin{equation} \label{hbsig}
 f(\frac{v+v_*}2 + \frac{|v-v_*|}2\,\sigma)\, f(\frac{v+v_*}2 - \frac{|v-v_*|}2\,\sigma)  =  f(v)\,  f(v_*) ,
\end{equation}
then for some $a\in\R$, $b\in \R^3$, $c>0$ 
\begin{equation}
\forall v \in \R^3, \qquad \ln f(v) = a + b \cdot v + c \, |v|^2 . 
\end{equation}
In this formulation, ``for all'' is to be replaced by ``for almost all'' if the function $f$ is not assumed to be continuous.
Noticing that $\bigg(\frac{v+v_*}2 + \frac{|v-v_*|}2\,\sigma, \frac{v+v_*}2 - \frac{|v-v_*|}2\,\sigma \bigg)$ is a parametrization of the set of $(v', v'_*)$ such that $v'+ v'_* = v + v_*$ and
$ |v|^2 + |v_*|^2 = |v'|^2 + |v_*'|^2$,  the assumption (\ref{hbsig})
can be replaced by: for all $v,v_*,v', v'_*\in\R^3$ such that 
\begin{equation} \label{hbs}
 v'+ v'_* = v + v_*, \qquad  |v|^2 + |v_*|^2 = |v'|^2 + |v_*'|^2,
\end{equation}
one assumes that
\begin{equation} \label{hbsig2}
 f(v')\, f(v'_*)  =  f(v)\,  f(v_*) .
\end{equation}
Note also that provided that $f>0$ a.e., denoting by $g =\ln f$, 
Boltzmann's H-theorem  amounts to proving that
the functions $g$ satisfying 
\begin{equation} \label{hbsigl2}
 g(v') + g(v'_*)  =  g(v) + g(v_*) 
\end{equation}
for all $v,v_*,v', v'_*\in\R^3$ such that (\ref{hbs}) holds
are exactly those of the form
\begin{equation}
g(v) = a + b \cdot v + c \, |v|^2 ,  
\end{equation}
where $a,c\in\R$, $b\in \R^3$ ($c>0$ if $g$ is assumed to be tending to $-\infty$ when $|v| \to +\infty$).
This last formulation~\eqref{hbsigl2} naturally appears when one considers the 
linearized Boltzmann kernel (around $M(v) := (2\pi)^{-3/2} \exp \bigg(- \frac{|v|^2}{2} \bigg)$): 
\begin{equation}
\label{eq:boltz_lin}
 Lg(v) = \int_{\R^3}\int_{\Sphere^2} M(v_*)\, \bigg\{ g(v'_*) + g(v') - g(v_*) - g(v) \bigg\}\, B\,d\sigma dv_*  ,  
\end{equation}
and the quantity
\begin{equation}
\label{eq:boltz_lin_int}
 \int_{\R^3}Lg(v)\, g(v)\, M(v) dv = -\frac14 \int_{\R^3}\int_{\R^3}\int_{\Sphere^2} M(v) M(v_*)\, \bigg[ g(v'_*) + g(v') - g(v_*) - g(v) \bigg]^2 \, B\, d\sigma dv_* dv .  
\end{equation}
It is therefore natural to ask that the formulation (\ref{hbsigl2}) holds in a space of functions $g$ which includes $L^2(\R^3; M\,dv)$.
Note that it is indeed the case, and that for all classical cross sections $B$, it can in fact be shown a quantitative version of the (linearized) H-theorem, enabling in some cases (for example hard spheres) to get a spectral gap (cf. \cite{baranger_mouhot}). 

We point out that the formulation~\eqref{hbsigl2} is also related to the equilibria of the quantum Boltzmann equation (see e.g.~\cite{EscMisVal03}). Indeed, in such cases~\eqref{hbsig2} has to be replaced by a more complicated expression of the form
\begin{equation*}
 f(v')\, f(v'_*)\, (1+\varepsilon f(v))\, (1+\varepsilon f(v_*))  =  f(v)\,  f(v_*)\, (1+\varepsilon f(v'))\, (1+\varepsilon f(v'_*)),
\end{equation*}
where $\varepsilon=-1$ and $\varepsilon=1$ respectively  correspond to Fermi-Dirac and Bose-Einstein statistics. By considering $g=\ln\frac{f}{1+\varepsilon f}$, we are indeed again brought back to~\eqref{hbsigl2}.
\medskip

In this paper, we are interested in situations where the conservation of energy $ |v|^2 + |v_*|^2 = |v'|^2 + |v_*'|^2$ 
is replaced by a more general relation
\begin{equation}
  \omega(v') + \omega(v_*') = \omega(v) + \omega(v_*),  
\end{equation}
where $\omega$ is a prescribed function from 
$\R^d$ to $\R$.
\medskip

Such a situation naturally appears when one considers the relativistic Boltzmann equation, with $\omega(p) = \sqrt{1+|p|^2}$ (where $p$ is the momentum). We refer to \cite{cercignani_rel} and \cite{strain} for details on the modeling 
of collisions in such a setting.
\medskip

It also appears when one looks at the four waves equation appearing in weak turbulence theory (cf. \cite{ZakLvoFal12}). In this theory, the operator modeling the interaction 
between four waves writes
$$ Q_W(f) (v)=\int W(v,v_*,v',v_*')\, \bigg[f(v')\,f(v_*')\,(f(v)+f(v_*))-f(v)\,f(v_*)\,(f(v')+f(v_*')) \,\bigg]  $$
\begin{equation} \label{opW}
\times\, \delta_{\{v+v_*=v'+v_*'\}}
\,\delta_{\{\omega(v)+\omega(v_*)=\omega(v')+\omega(v_*')\}}\,
 dv_*dv_*'dv'.
\end{equation}
A precise meaning can be given to the Dirac functions used above (cf. \cite{germ}), and $W$ plays the role of $B$ in the Boltzmann equation. In this context the variable $v$ represents a wave vector, which is more usually denoted $k$ or $p$, however to keep a uniform notation throughout the paper we stick with the letter $v$.
 \par
 A typical value for $\omega$ is  $\omega(v)=C\, |v|^{\alpha}$, for $0<\alpha<1$ and $C>0$. In particular, in the two-dimensional case $d=2$,  $\omega(v)=C\sqrt{|v|}$ is used to describe gravitational waves on a fluid surface (see~\cite{ZakLvoFal12}).
\medskip

An H-theorem also holds for four waves kinetic equations, with the entropy $\int \ln f dv$ and the entropy production $ \int Q_W(f) \, f^{-1} \,dv $. Indeed, one can compute
$$\int  Q_W(f) (v)\, f^{-1}(v)\, dv =\frac{1}{4}\int W(v,v_*,v',v_*')\, \bigg[f^{-1}(v) + f^{-1}(v_*) -f^{-1}(v') -f^{-1}(v_*')\bigg]^2  $$
\begin{equation} \label{opWDE}
\times\, f(v)f(v_*)f(v')f(v_*')\delta_{\{v+v_*=v'+v_*'\}}
\,\delta_{\{\omega(v)+\omega(v_*)=\omega(v')+\omega(v_*')\}}\,
 dvdv_*dv_*'dv'.
\end{equation}
When $W>0$, denoting $g = f^{-1}$, we see that the (strictly positive) continuous equilibria of $Q_W$ are the function $g$ satisfying
\begin{equation} \label{cup}
 g(v') + g(v_*')  =  g(v) + g(v_*) 
\end{equation}
for all $v,v_*,v', v_*'\in\R^d$ such that 
\begin{equation} \label{cus} 
 v+v_*=v'+v_*', \qquad 
\omega(v)+\omega(v_*)=\omega(v')+\omega(v_*') . 
\end{equation}
If one looks for equilibria which are $L^1_{loc}$ (and a.e. strictly positive), then (\ref{cup}) has to be satisfied for a.e.  $v,v_*,v', v_*'\in\R^d$ such that (\ref{cus}) holds.

It is obvious that all functions $g$ of the form
\begin{equation} \label{ffom}
g(v) = a + b \cdot v + c \, \omega(v) ,  
\end{equation}
satisfy (\ref{cup}). Our goal is to prove the reciprocal of this statement. In order to keep a coherence with the linearized theory, we 
wish to have a proof which holds when $g$ is only supposed to belong to $L^2$ (possibly with a weight). In practice, we shall show the result when $g \in L^1_{loc}$.
\medskip

We intend to impose as few assumptions as possible on $\omega$. Indeed, as previously stated, typical $\omega$ appearing in weak turbulence theory can be singular at point $0$. Denoting here (and for the rest of the paper) $\Rd := \R^d\setminus\{0\}$, we assume therefore that $\omega \in 
\CC^2(\Rd,\R)$. 
\medskip

We can now present the first version of our main Theorem, corresponding to the case when $g$ is assumed to have some smoothness.
\medskip

\begin{theorem}
\label{th:22}
Let $d\in\{2,3\}$ and $\omega\in\CC^2(\Rd,\R)$. Assume that there exists $i,j\in\{1,\ldots,d\}$, $i\neq j$, such that
\begin{equation}
\label{hyp:linear_indep}
\left\{1,\partial_i\omega,\partial_j\omega\right\} \ \text{are linearly independant in }\CC^1(\Rd).
\end{equation} 
Let $g\in \CC^1(\Rd,\R)$ satisfying~\eqref{cup}, \eqref{cus}. 
\par 
Then, there exist $a,c\in\R$ and $b\in\R^d$ such that, for all $v$ in $\Rd$,
\begin{equation*}
g(v)=a+b\cdot v +c\,\omega(v).
\end{equation*}
\end{theorem}

However, we wish to relax the smoothness assumption for  $g$  in the previous statement, in order to have a consistency with the setting naturally appearing in linearized formulations of the collision kernels. Unfortunately, assumption (\ref{cup}), (\ref{cus}) is not directly adapted to treat 
functions defined up to their value on a zero-measure set. We therefore first state a lemma enabling to define properly the problem when $g \in L^1_{loc}(\R^d)$:

\begin{lemma}
\label{lem:para_var} 
Let $d\geq 2$ and $\omega\in\CC^2(\Rd,\R)$. We define
\begin{equation}\label{defa}
A := \left\{(v,v_*)\in \left(\Rd\right)^2,\ \nabla\omega(v) \neq \nabla\omega(v_*)\right\}.
\end{equation}
For every $(\bx,\by)\in A$, there exists a bounded neighborhood $U\subset A$ of $(\bx,\by)$, a neighborhood $V\subset\R^{d-1}$ of $0$ and a function $\psi=\psi(v,v_*,\sigma)\in\CC^2(\R^d\times\R^d\times\R^{d-1},\R^d)$ such that, for all $(v,v_*)$ in $U$ and all $\sigma$ in $V$,
\begin{equation}
\label{eq:def_psi}
\psi(v,v_*,0)=0 \quad\text{and}\quad \omega(v)+\omega(v_*)=\omega(v-\psi(v,v_*,\sigma))+\omega(v_*+\psi(v,v_*,\sigma)),
\end{equation}
and such that, for all $(v,v_*)$ in $U$,
\begin{equation}
\label{eq:psi_subm}
\rank(D_\sigma\psi(v,v_*,0))=d-1.
\end{equation}
\end{lemma}

We can now state the second version of our theorem, where $g$ is only assumed to be locally integrable. 

\begin{theorem}
\label{th:22version2}
Let $d\in\{2,3\}$ and $\omega\in\CC^2(\Rd,\R)$. Assume that there exists $i,j\in\{1,\ldots,d\}$, $i\neq j$, such that
\begin{equation}
\label{hyp:linear_indep_rep}
\left\{1,\partial_i\omega,\partial_j\omega\right\} \ \text{are linearly independant in }\CC^1(\Rd).
\end{equation} 
Let $g\in L^1_{loc}(\R^d,\R)$ such that for  all $(\bx,\by)\in A$, for almost every $(v,v_*)$ in $U$ and almost every $\sigma\in V$,
\begin{equation} \label{cond:22}
g(v)+g(v_*)=g(v-\psi(v,v_*,\sigma))+g(v_*+\psi(v,v_*,\sigma)),
\end{equation}
where $A$, $U$, $V$ and $\psi$ are defined in Lemma~\ref{lem:para_var}. Assume also that the boundary of $A$ in $\left(\Rd\right)^2$ is of Lebesgue measure zero (we recall that $A$ is defined by (\ref{defa})).
\par 
 Then, there exist $a,c\in\R$ and $b\in\R^d$ such that, for a.e. $v$ in $\Rd$,
\begin{equation*}
g(v)=a+b\cdot v +c\,\omega(v).
\end{equation*}
\end{theorem}
\medskip

In the remark below, we present counterexamples which show the limitations 
of our Theorem.

\begin{remark}
\label{rem:th22}
We first note that in dimension $d=1$, forgetting about the assumption~\eqref{hyp:linear_indep_rep} which is meaningless in this case, the dispersion law $\omega(v)=\sqrt{1+v^2}$ provides a counterexample, since the function $\argsinh$ satisfies~\eqref{cond:22}~\cite{BalFer98}.
\par 
In dimension $d\geq 2$, assumption~\eqref{hyp:linear_indep_rep} is really crucial. Indeed, in dimension $d=2$, the smooth functions $\omega$ that do not satisfy~\eqref{hyp:linear_indep_rep} are, up to a rotational change of coordinates, of the form
\begin{equation}
\label{eq:cex22}
\omega(v_1,v_2) = \alpha h(v_1) + \beta v_2, 
\end{equation} 
where $\alpha,\beta\in\R$ and $h$ is any smooth function from $\R$ to $\R$. One can show that, if $\omega$ is of the form~\eqref{eq:cex22}, with a fuction $h$ that is strictly concave or convex (this condition is merely sufficient here, but not necessary), the set of $(v,v_*,v',v_*')$ satisfying (\ref{cus}) is exactly 
\begin{equation*}
\left\{(v,v_*,v,v_*) \ |\ (v,v_*)\in\left(\R^2\right)^2 \right\}\cup \left\{(v,v_*,v_*,v) \ |\ (v,v_*)\in\left(\R^2\right)^2 \right\}.
\end{equation*}
Therefore assumption~\eqref{cond:22} is trivially satisfied by \emph{any} function $g:\R^2\to \R$. 
\par 
In higher dimensions such counterexamples persist, and more could probably be found. Roughly speaking, assumption~\eqref{hyp:linear_indep_rep} ensures that the set of $(v,v_*,v',v_*')$ satisfying (\ref{cus})  is sufficiently \emph{non degenerate} and provides enough constraints so that $g$ is really constrained by $\omega$.
\end{remark}

\medskip

In the case of the Boltzmann equation (that is, $\omega(v) = |v|^2$), the proof of Thm. \ref{th:22} (that is, the case of equality of the  
the H-Theorem) is due to Boltzmann himself, provided that $g$ has some extra regularity (typically, two derivatives of $g$ are assumed to exist). 
It is based on the use of successive well-chosen differential operators, and uses the fact that the final form of $g$ being a polynomial of degree $2$, all its second order derivatives are constants.
Taking derivatives in the sense of distributions instead of taking them in the usual sense, it is not difficult to recover Thm. \ref{th:22version2}
in this case.
We refer for example to \cite{LD89} and \cite{wennberg} for a description of the proof in such a setting. 
\par 
For the relativistic Boltzmann equation, the situation is somewhat more intricate, but the approach based on the use of successive differential operators still works, and a mathematically rigorous proof (when $g$ is of class $\CC^2$) is available in \cite{cercignani_rel}. There is also suggested how to treat the case of non smooth $g$ by a suitable use of the theory of distributions.
\par 
In the general case that we now examine, we propose a proof which is based on significantly different principles. Indeed one is limited by the possible  singularity at point $0$ of $\omega$ when one wants to take derivatives and, more importantly, one needs to use somewhere in the proof the assumption (\ref{hyp:linear_indep_rep}), which, according to Remark \ref{rem:th22}, is close to a necessary condition (at least in dimension $2$) for obtaining the result. Such an assumption is reminiscent of the computations performed when estimating the Landau entropy dissipation in \cite{landau_jfa} and \cite{landau_braga}, cf. also \cite{strain} in the relativistic case. We shall therefore use a method of proof which makes use of the equivalent of the Landau equation when one deals with collision operators coming out of the theory of weak turbulence.
\par 
 More precisely, we shall first show that, in a weak sense, 
the functions $g$ satisfying (\ref{cup}), (\ref{cus}) are such that
\begin{equation}\label{lst}
(\nabla g(v) - \nabla g(v_*)) \times   (\nabla \omega(v) - \nabla \omega(v_*)) = 0. 
\end{equation}
This statement is easy to prove when both $\omega$ and $g$ are smooth. In the situation that we consider, the proof is unfortunately rather technical since all statements have to be understood in a weak sense. 
A second part of the proof consists in using suitable integral operators rather than differential operators. Note that this is the strategy adopted in \cite{landau_vill2}, \cite{landau_jfa} and \cite{landau_braga}. At this point, it 
significantly differs from the historical strategy consisting in differentiating formula (\ref{lst}) one or two times. Finally, the formulas obtained after applying the integral operators enable to write down a linear system which can be solved by Cramer's formula provided that 
some determinant is not $0$, and this condition in fact  corresponds  
to assumption (\ref{hyp:linear_indep_rep}).
\medskip

We finally provide some results on the so-called three-waves collision kernel appearing in weak turbulence theory~\cite{ZakLvoFal12}.  It writes
\begin{equation*}
\tilde Q_{\tilde W}(f) (v) = \int \left(R(v,v',v'')-R(v',v,v'')-R(v'',v,v')\right) dv'dv'',
\end{equation*}
where
\begin{equation*}
R(v,v',v'') = \tilde W(v,v',v'') \left(f(v')f(v'')-f(v)(f(v')+f(v''))\right) \delta_{\{v=v'+v''\}} \delta_{\{\omega(v)=\omega(v')+\omega(v'')\}},
\end{equation*}
and the nonnegative coefficient $\tilde W$ satisfies $\tilde W(v,v',v'')=\tilde W(v,v'',v')$. The associated entropy production is given by
\begin{align*}
&\int \tilde Q_{\tilde W}(f) (v) f^{-1}(v)dv = \\
&\qquad \int \tilde W (v,v',v'')\, f(v)\,f(v')\,f(v'')\,
\left(\frac{1}{f(v)}-\frac{1}{f(v')}-\frac{1}{f(v'')}\right)^2 \\
&\qquad\qquad\qquad\qquad \times \delta_{\{v=v'+v''\}} \delta_{\{\omega(v)=\omega(v')+\omega(v'')\}}dvdv'dv''.
\end{align*}
Considering strictly positive continuous functions $f$ and denoting $g=f^{-1}$, assuming moreover that $\tilde{W}>0$,
the case of equality in the H-theorem consists in looking for the functions $g: \R^d \to \R$ satisfying the following condition:
\begin{equation} \label{cond:21}
\omega(v')+\omega(v'')=\omega(v'+v'') \quad\Rightarrow\quad  g(v')+g(v'')=g(v'+v''),
\end{equation}
for all $(v',v'')\in \left(\R^d\right)^2$. It is obvious that for any 
$b\in\R^d$ and $c\in\R$, the function defined by
\begin{equation}
g(v): =b\cdot v + c\,\omega(v)
\end{equation}
satisfies condition (\ref{cond:21}). Our next result is devoted to establishing the 
reciprocal of this statement. Unfortunately, we shall need more stringent assumptions on $\omega$ and $g$ than in the four-waves case. More precisely, we can write the following theorem: 

\begin{theorem}
\label{th:21}
Let $d\geq 2$ and $\omega\in\CC^1(\R^d,\R)$ such that $\omega(0)=0$, $\nabla\omega(0)=0$, $\omega(v)>0$ for all $v\neq 0$ and $\nabla\omega(v)\neq 0$ for all $v\neq 0$. Assume also that, for all $a\in\R$, the set $\omega^{-1}\left(\{a\}\right)$ is connected. Let $g\in\CC^1(\R^d,\R)$ satisfying~\eqref{cond:21}. Then, there exist $b\in\R^d$ and $c\in\R$ such that, for all $v$ in $\R^d$,
\begin{equation}
g(v) =b\cdot v + c\,\omega(v) .
\end{equation}
\end{theorem}

Note that the disappearance of the constant term in $g$ (compared to Theorem~\ref{th:22} and Theorem~\ref{th:22version2}) comes from the fact that three-waves 
 interactions do not preserve the mass.

\begin{remark}
For some functions $\omega$, the set $\Gamma := \{(v',v''), \,\, \omega(v'+v'') = \omega(v') + \omega(v'') \}$ is reduced to
\begin{equation*}
\left\{(v,0,v) \ |\ v\in\R^d \right\}\cup \left\{(0,v,v) \ |\ v\in\R^d \right\}.
\end{equation*}
In such cases, assumption~\eqref{cond:21} is trivially satisfied by any function $g:\R^d \to \R$. In Theorem~\ref{th:21}, this degeneracy is mainly prevented by the smoothness assumption on $\omega$. Indeed, typical examples of dispersion laws leading to a degenerate set $\Gamma$ 
are of the form $\omega(v)= |v|^\beta$, $\beta<1$, which are excluded from Theorem~\ref{th:21} since they are not differentiable at point $0$.

For the limiting case $\omega(v)= |v|$ in dimension $2$, which is \emph{almost} $\CC^1$ (in the sense that it belongs to $W^{1,\infty}_{loc}$), it is easy to see that any function $g$ taking, in polar coordinates, the form $g(r,\theta)=rh(\theta)$, satisfies~\eqref{cond:21}. Therefore the $\CC^1$ assumption is somewhat close to optimal.

Another counterexample of particular interest is given (still in dimension $d=2$),  by
\begin{equation*}
\omega(v) = \frac{v_1}{1+v_1^2+v_2^2},
\end{equation*}
which is the dispersion law corresponding to Rossby waves. This dispersion law is known to be degenerate, meaning that it admits an extra invariant~\cite{Bal91}
\begin{equation*}
g(v)=\arctan\frac{v_1\sqrt 3+v_2}{v_1^2+v_2^2}-\arctan\frac{-v_1\sqrt 3+v_2}{v_1^2+v_2^2},
\end{equation*} 
satisfying~\eqref{cond:21}. Note that several assumptions of Theorem~\ref{th:21} are not satisfied by $g$ and $\omega$ in this counterexample. 
\end{remark}
\medskip

The proof of Theorem \ref{th:21} follows roughly the same lines as that of Theorem \ref{th:22version2}, but many details are different. A first significant intermediary result consists in establishing that $g$ is in fact of the form $g(v) = \nabla g(0)\cdot v + \mu(\omega(v))$, where $\mu$ is a function with a certain amount of smoothness. This is done by using a differentiation with respect to a variable which parametrizes locally the set of $v, z$ such that $\omega(v) + \omega(z) = \omega (v+z)$. It remains then to identify $\mu$: this is done by writing a functional equation for $\mu$ and by using again some differentiation.

\bigskip 

The rest of the paper is structured as follows: In section \ref{sec2}, we prove our main theorem (that is, Theorem \ref{th:22version2}). Then, section \ref{sec3} is devoted to the proof of Theorem \ref{th:21}.

\section{Proof of the main result} \label{sec2}

In this section we prove Theorem~\ref{th:22} and Theorem~\ref{th:22version2}.

\medskip

We start with the proof of Lemma~\ref{lem:para_var}, which enables to define properly the assumptions of Theorem~\ref{th:22version2}.
\medskip 

\noindent\textit{Proof of Lemma \ref{lem:para_var}.}
We consider the function \begin{equation*}
\Phi:\left\{
\begin{aligned}
&\R^d\times\R^d\times\R^d \to \R ,\\
&(v,v_*,z) \mapsto \omega(v)+\omega(v_*)-\omega(v-z)-\omega(v_*+z) .
\end{aligned}
\right.
\end{equation*}
Since $(\bx,\by)\in A$, we see that
\begin{equation*}
\nabla_z \Phi(\bx,\by,0)=\nabla\omega(\bx)-\nabla\omega(\by) \neq 0.
\end{equation*}
Thus, there exists $i\in\{1,\ldots,d\}$ such that $\pa_i\omega(\bx)-\pa_i\omega(\by)\neq 0$, and since $\Phi(\bar{v}, \bar{v_*},0) = 0$, thanks to the implicit function theorem, there exists a neighborhood $U\subset A$ of $(\bx,\by)$, a neighborhood $V\subset\R^{d-1}$ of $0$, and a function $h\in \CC^2(U\times V,\R)$ such that, for all $(v,v_*,z)$ in a neighborhood of $(\bx,\by,0)$,
\begin{equation*}
\Phi(v,v_*,z)=0 \quad \Leftrightarrow \quad  z_i=h(v,v_*,z_{i'}),
\end{equation*}
where $z_{i'}=(z_1,\ldots,z_{i-1},z_{i+1},\ldots,z_d)$. One can then consider the function $\psi:U\times V \to \R$ defined component-wise by
\begin{equation*}
\psi_j(v,v_*,\sigma) = \left\{
\begin{aligned}
&\sigma_j \quad &j<i ,\\
&h(v,v_*,\sigma) \quad &j=i ,\\
&\sigma_{j-1} \quad &j>i ,
\end{aligned}
\right.
\end{equation*}
and the lemma is proven. \hfill $\qed$
\medskip

We now turn to the proof of Theorem~\ref{th:22version2}. We start with the 

\begin{lemma}
\label{lem:diffeo}
Let $d\geq 2$ and $\omega\in\CC^2(\Rd,\R)$. Let $(\bx,\by)\in A$ and consider $U$, $V$ and $\psi$ as in Lemma~\ref{lem:para_var}. Let $\gamma_\sigma:\R^d\times\R^d \to \R^d\times\R^d$ defined as
\begin{equation*}
\gamma_{\sigma}(v,v_*)=\left(v-\psi(v,v_*,\sigma),v_*+\psi(v,v_*,\sigma)\right).
\end{equation*}
To shorten some expressions we also introduce
\begin{equation*}
X=(v,v_*),\quad \Psi(X,\sigma)=(-\psi(v,v_*,\sigma),\psi(v,v_*,\sigma)),
\end{equation*}
so that
\begin{equation*}
\gamma_\sigma(X)=X+\Psi(X,\sigma).
\end{equation*}
For all neighborhood $U_1\subset A$ of $(\bx,\by)$ such that $\bar U_1\subset U$, there exists $\sigma_0>0$ such that, for all $\Vert\sigma\Vert\leq \sigma_0$, $\gamma_\sigma$ is a $\CC^1$ diffeomorphism on $U_1$. Besides, $\gamma_0=Id$ and, for all $X\in U_1$, both $\sigma\mapsto \gamma_\sigma^{-1}(X)$ and $\sigma\mapsto D\gamma_\sigma^{-1}(X)$ are differentiable at $0$, and 
\begin{equation}
\label{eq:d_sigma}
\frac{\partial}{\partial\sigma}\Bigr|_{\sigma=0}\gamma_\sigma^{-1}(X) = -D_\sigma\Psi(X,0), \quad 
\frac{\partial}{\partial\sigma}\Bigr|_{\sigma=0}D\gamma_\sigma^{-1}(X) = -D^2_{X,\sigma}\Psi(X,0).
\end{equation}
\end{lemma}
\begin{proof}
Since $\Psi(X,0)=0$ for all $X$ in $U$, we immediately get $\gamma_0=Id$. Therefore, using the uniform continuity, there exists $\sigma_0>0$ such that, for all $\Vert\sigma\Vert\leq \sigma_0$ and all $X\in U_1$,
\begin{equation}
\label{eq:DXPsi}
\left\Vert D_X\Psi(X,\sigma) \right\Vert \leq \frac{1}{2}.
\end{equation}
Recalling that $\gamma_\sigma=Id+\Psi(\cdot,\sigma)$, we infer that $D\gamma_\sigma(X)$ is invertible for all $\Vert\sigma\Vert\leq \sigma_0$ and all $X\in U_1$, with
\begin{equation}
\label{eq:Dgamma_inv}
D\gamma_\sigma(X)^{-1} = \sum_{k=0}^\infty (-1)^k \left(D_X\Psi(X,\sigma)\right)^k.
\end{equation}
Besides, $\gamma_\sigma$ is injective on $U_1$ by~\eqref{eq:DXPsi}, therefore $\gamma_\sigma$ is indeed a $\CC^1$ diffeomorphism on $U_1$ for all $\Vert\sigma\Vert\leq \sigma_0$.

Next, we fix $X\in U_1$ and show that $\sigma\mapsto \gamma_\sigma^{-1}(X)$ is continuous at point $0$. Since $\gamma_0=Id$, for $\sigma$ small enough, we see that $X\in\gamma_\sigma (U_1)$, and we can estimate
\begin{align*}
\left\Vert \gamma_\sigma^{-1}(X) - \gamma_0^{-1}(X) \right\Vert &= \left\Vert \gamma_\sigma^{-1}(X) - X \right\Vert \\
&= \left\Vert \gamma_\sigma^{-1}(\gamma_\sigma(X) +(\gamma_0(X)-\gamma_\sigma(X))) - X \right\Vert \\
&= \left\Vert \gamma_\sigma^{-1}(\gamma_\sigma(X)) + D\gamma_\sigma^{-1}(\gamma_\sigma(X))(\gamma_0(X)-\gamma_\sigma(X))+ o(\gamma_0(X)-\gamma_\sigma(X)) - X \right\Vert \\
&\leq \left\Vert D\gamma_\sigma^{-1}(\gamma_\sigma(X))\right\Vert \left\Vert \gamma_0(X)-\gamma_\sigma(X)\right\Vert + o(\left\Vert\gamma_0(X)-\gamma_\sigma(X)\right\Vert).
\end{align*}
Since $\left\Vert D\gamma_\sigma^{-1}(\gamma_\sigma(X))\right\Vert$ is bounded uniformly in $\sigma$ for $\sigma$ small enough (for instance by $2$ for $\Vert\sigma\Vert \leq \sigma_0$), we get that $\sigma\mapsto \gamma_\sigma^{-1}(X)$ is continuous at point $0$. To then show that it is differentiable (with the announced derivative), we introduce
\begin{equation*}
g(\sigma)=\gamma_\sigma^{-1}(X)-\gamma_0^{-1}(X) + D_\sigma\Psi(X,0)\sigma.
\end{equation*}
We have just shown that $g(\sigma)$ is going to $0$ when $\sigma$ goes to $0$, and we now have to prove that $g(\sigma)=o(\Vert\sigma\Vert)$. We start from
\begin{align*}
\gamma_\sigma^{-1}(X) = X - D_\sigma\Psi(X,0)\sigma + g(\sigma),
\end{align*}
compose with $\gamma_\sigma$, and use a first order Taylor expansion around $\gamma_0(X)$, and obtain
\begin{align*}
X &= \gamma_\sigma\left(X - D_\sigma\Psi(X,0)\sigma + g(\sigma)\right) \\
&= \gamma_0(X) +D\gamma_0 (X) \left(- D_\sigma\Psi(X,0)\sigma + g(\sigma)\right) + \frac{\partial}{\partial\sigma}\Bigr|_{\sigma=0}\gamma_0(X)\sigma +o\left(\Vert\sigma\Vert + \left\Vert - D_\sigma\Psi(X,0)\sigma + g(\sigma) \right\Vert\right) \\
&= X - D_\sigma\Psi(X,0)\sigma + g(\sigma) + D_\sigma\Psi(X,0)\sigma +o\left(\Vert\sigma\Vert + \left\Vert - D_\sigma\Psi(X,0)\sigma + g(\sigma) \right\Vert\right).
\end{align*}
Therefore, we end up with
\begin{align*}
g(\sigma) &= o\left(\Vert\sigma\Vert + \left\Vert - D_\sigma\Psi(X,0)\sigma + g(\sigma) \right\Vert\right)\\
&=o(\Vert\sigma\Vert),
\end{align*}
which proves that $\frac{\partial}{\partial\sigma}\Bigr|_{\sigma=0}\gamma_\sigma^{-1}(X)$ exists and is equal to $-D_\sigma\Psi(X,0)$. The existence of $\frac{\partial}{\partial\sigma}\Bigr|_{\sigma=0}D\gamma_\sigma^{-1}(X)$ and the fact that it is equal to $-D^2_{X,\sigma}\Psi(X,0)$ is more straightforward and can be obtained directly from~\eqref{eq:Dgamma_inv}, using that $\Psi$ is twice differentiable and that $D_X\Psi(X,0)=0$.
\end{proof}

We now prove an intermediary result, which is a key ingredient in our proof of Theorem~\ref{th:22version2}. 

\begin{proposition}
\label{prop:weak_form}
Let $d\in\{2,3\}$ and $\omega\in\CC^2(\Rd,\R)$. Let $g\in L^1_{loc}(\R^d,\R)$ satisfying~\eqref{cond:22}, with the notations of Lemma \ref{lem:para_var}. Assume (still using those notations) that the boundary of $A$ is of Lebesgue measure zero. Then, for all $\varphi\in\CC^1_c(\left(\Rd\right)^2,\R)$,
\begin{align*}
\int_{\left(\Rd\right)^2}  \left(g(v)\nabla_v\varphi(v,v_*)-g(v_*)\nabla_{v_*}\varphi(v,v_*)\right)\times\left(\nabla\omega(v)-\nabla\omega(v_*)\right)dv dv_*=0.
\end{align*}
\end{proposition}

\begin{remark}
In the above statement, if $a,b\in\R^d$ with $d=2$, the cross product $a\times b$ must be understood as $a_1b_2-a_2b_1$.
\end{remark}

We start the proof of this proposition with two lemmas. 

\begin{lemma}
\label{lem:weak_form}
Let $d\in\{2,3\}$ and $\omega\in\CC^2(\Rd,\R)$. Let $(\bx,\by)\in A$ and consider $U$, $V$ and $\psi$ as in Lemma~\ref{lem:para_var}. Let $g\in L^1_{loc}(\R^d,\R)$ such that, for almost every $(v,v_*)$ in $U$ and almost every $\sigma\in V$,
\begin{equation*}
g(v)+g(v_*)=g(v-\psi(v,v_*,\sigma))+g(v_*+\psi(v,v_*,\sigma)).
\end{equation*}
Then, for all $\varphi\in\CC^1_c(U,\R)$,
\begin{align*}
\int_{U}  \left(g(v)\nabla_v\varphi(v,v_*)-g(v_*)\nabla_{v_*}\varphi(v,v_*)\right)\times\left(\nabla\omega(v)-\nabla\omega(v_*)\right)dv dv_*=0.
\end{align*}
\end{lemma}
\begin{proof}
We start by introducing $\gamma_\sigma$ as defined in Lemma~\ref{lem:diffeo}, together with $G(X)=g(v)+g(v_*)$ (still using the notations of Lemma~\ref{lem:diffeo}). We have, for almost every $X$ in $U$ and almost every $\sigma$ in V,
\begin{equation*}
G(X)=G(\gamma_\sigma(X)).
\end{equation*}
Integrating against $\varphi\in\CC^1_c(U,\R)$, we get
\begin{equation*}
\int_U G(X)\varphi(X) dX = \int_U G(\gamma_\sigma(X)) \varphi(X) dX.
\end{equation*}
Next, we use the fact that, for $\sigma$ small enough, $\gamma_\sigma$ is a $\CC^1$ diffeomorphism. More precisely, thanks to Lemma~\ref{lem:diffeo}, we can consider $\sigma_0>0$ and an open subset $U_1$ of $A$ satisfying
\begin{equation*}
\supp(\varphi) \subset U_1 \subset \bar U_1\subset U,
\end{equation*}
and such that $\gamma_\sigma$ is a $\CC^1$ diffeomorphism on $U_1$, for all $\Vert\sigma\Vert\leq \sigma_0$. Using further that $\gamma_0=Id$, up to taking $\sigma_0$ small enough, we can assume that for all $\Vert\sigma\Vert\leq \sigma_0$,
\begin{equation*}
\gamma_\sigma\left(\supp(\varphi)\right) \subset U_1.
\end{equation*} 
Therefore, we obtain for almost every $\Vert\sigma\Vert\leq \sigma_0$,
\begin{align}
\label{eq:chngt_var}
\int_{U_1} G(X)\varphi(X) dX &= \int_{U_1} G(\gamma_\sigma(X)) \varphi(X) dX \nonumber\\
&= \int_{\gamma_\sigma(U_1)} G(X) \varphi(\gamma_\sigma^{-1}(X)) \left\vert \det(D\gamma_\sigma^{-1}(X)) \right\vert dX \nonumber\\
&= \int_{U_1} G(X) \varphi(\gamma_\sigma^{-1}(X)) \left\vert \det(D\gamma_\sigma^{-1}(X)) \right\vert dX.
\end{align}
Notice that the r.h.s is a $\CC^1$ function of $\sigma$, and thus~\eqref{eq:chngt_var} holds for all $\Vert\sigma\Vert\leq \sigma_0$. Differentiating~\eqref{eq:chngt_var} with respect to $\sigma$, evaluating at $\sigma=0$ and using~\eqref{eq:d_sigma}, we obtain
\begin{align*}
0 &=\int_{U_1} G(X) \left[D\varphi\left(X\right)D_\sigma\Psi(X,0)
 + \varphi\left(X\right) \tr\begin{pmatrix}
D^2_{X,\sigma}\Psi(X,0)
\end{pmatrix}\right]dX \nonumber\\
&=\int_{U_1} G(X) \left[D\varphi\left(X\right)D_\sigma\Psi(X,0)
 + \varphi\left(X\right) \diverg_X\left(D_\sigma\Psi(X,0)\right)\right]dX \nonumber\\
&= \int_{U}  G(X) \diverg_X\left(\varphi(X)D_\sigma\Psi(X,0)\right) dX.
\end{align*}
This can be rewritten as
\begin{align*}
 &\int_{U}  \left(g(v)+g(v_*)\right)\left[\diverg_v\left(\varphi(v,v_*)D_\sigma\psi(v,v_*,0)\right) -\diverg_{v_*}\left(\varphi(v,v_*)D_\sigma\psi(v,v_*,0)\right)\right] dv dv_* =0,
\end{align*}
and then simplified into
\begin{align}
\label{eq:forme_faible_derivee}
&\int_{U}  \left[g(v)\diverg_v\left(\varphi(v,v_*)D_\sigma\psi(v,v_*,0)\right)  -g(v_*)\diverg_{v_*}\left(\varphi(v,v_*)D_\sigma\psi(v,v_*,0)\right)\right] dv dv_* = 0,
\end{align}
which is nothing but the weak (and local) formulation of 
\begin{equation*}
\left(Dg(v)-Dg(v_*)\right)D_\sigma\psi(v,v_*,0)=0.
\end{equation*}

We now consider the case $d=3$ (the case $d=2$ being similar but simpler). We introduce $e_1,e_2:\R^3\times\R^3 \to \R^3$ which denote the columns of $D_\sigma\psi(v,v_*,0)$, that is
\begin{equation*}
D_\sigma\psi(v,v_*,0) = \begin{pmatrix}
e_1(v,v_*) & e_2(v,v_*)
\end{pmatrix}.
\end{equation*}
Thanks to~\eqref{eq:def_psi}, we know that
\begin{equation*}
\left(D\omega(v)-D\omega(v_*)\right)D_\sigma\psi(v,v_*,0)=0,
\end{equation*}
and since $e_1(v,v_*)$ and $e_2(v,v_*)$ are independent thanks to~\eqref{eq:psi_subm}, and of class $\CC^1$ since $\psi$ is of class $\CC^2$, there exists a function $\lambda\in\CC^1(U,\R_*)$ such that for all $(v,v_*)$ in $U$,
\begin{equation}
\label{eq:lambda}
e_1(v,v_*) \times e_2(v,v_*) = \lambda(v,v_*) \left(\nabla\omega(v)-\nabla\omega(v_*)\right).
\end{equation}
Using~\eqref{eq:forme_faible_derivee}, we see that, for all $i\in\{1,2\}$ and all $\varphi\in\CC^1_c(U,\R)$,
\begin{equation*}
\int_{U}  \left[g(v)\diverg_v\left(\varphi(v,v_*)e_i(v,v_*)\right)-g(v_*)\diverg_{v_*}\left(\varphi(v,v_*)e_i(v,v_*)\right)\right] dv dv_* = 0.
\end{equation*}
Since each $e_j$ is of class $\CC^1$, we can consider any component of $\varphi e_j$ instead of $\varphi$ in the above identity, which we rewrite in a more compact form using the tensor product notation, for all $i,j\in\{1,2\}$ and all $\varphi\in\CC^1_c(U,\R)$,
\begin{align*}
&\int_{U}  \left[g(v)\diverg_v\left(\varphi(v,v_*)e_j(v,v_*)\otimes e_i(v,v_*)\right) - g(v_*)\diverg_{v_*}\left(\varphi(v,v_*)e_j(v,v_*)\otimes e_i(v,v_*)\right)\right] dv dv_* = 0,
\end{align*}
which then yields
\begin{align*}
&\int_{U}  \left[g(v)\diverg_v\left(\varphi(v,v_*)\left(e_2(v,v_*)\otimes e_1(v,v_*)-e_1(v,v_*)\otimes e_2(v,v_*)\right)\right)\right. \nonumber\\
&\hspace{1.5cm } \left. - g(v_*)\diverg_{v_*}\left(\varphi(v,v_*)\left(e_2(v,v_*)\otimes e_1(v,v_*)-e_1(v,v_*)\otimes e_2(v,v_*)\right)\right)\right] dv dv_* = 0.
\end{align*}
Making use of the vectorial calculus formula
\begin{equation*}
\diverg(b\otimes a-a\otimes b)=\rot(a\times b),
\end{equation*}
and combining it with~\eqref{eq:lambda}, we end up with
\begin{align*}
&\int_{U}  \left[g(v)\rot_v\left(\varphi(v,v_*)\lambda(v,v_*) \left(\nabla\omega(v)-\nabla\omega(v_*)\right)\right)\right. \nonumber\\
&\hspace{1.5cm } \left. - g(v_*)\rot_{v_*}\left(\varphi(v,v_*)\lambda(v,v_*) \left(\nabla\omega(v)-\nabla\omega(v_*)\right)\right)\right] dv dv_* = 0.
\end{align*}
Since this identity holds for all $\varphi\in\CC^1_c(U,\R)$ and since $\lambda$ is of class $\CC^1$ and does not vanish, we obtain, for all $\varphi\in\CC^1_c(U,\R)$,
\begin{align*}
0 &= \int_{U}  \left[g(v)\rot_v\left(\varphi(v,v_*)\left(\nabla\omega(v)-\nabla\omega(v_*)\right)\right) - g(v_*)\rot_{v_*}\left(\varphi(v,v_*)\left(\nabla\omega(v)-\nabla\omega(v_*)\right)\right)\right] dv dv_* \\
 &= \int_{U}  \left(g(v)\nabla_v\varphi(v,v_*)-g(v_*)\nabla_{v_*}\varphi(v,v_*)\right)\times\left(\nabla\omega(v)-\nabla\omega(v_*)\right)dv dv_*,
\end{align*}
which is nothing but the weak (and local) formulation of
\begin{equation*}
\left(\nabla g(v) -\nabla g(v_*)\right)\times  \left(\nabla \omega(v) -\nabla \omega(v_*)\right) = 0.
\end{equation*}
\end{proof}

\begin{lemma}
\label{lem:partition_unite}
Let $d\geq 2$, $\omega\in\CC^2(\Rd,\R)$ and $g\in L^1_{loc}(\R^d,\R)$. Assume that, for all $(\bx,\by)\in A$, there exists a neighborhood $U\subset A$ of $(\bx,\by)$ such that, for all $\varphi\in\CC^1_c(U,\R)$, 
\begin{align*}
\int_{U}  \left(g(v)\nabla_v\varphi(v,v_*)-g(v_*)\nabla_{v_*}\varphi(v,v_*)\right)\times\left(\nabla\omega(v)-\nabla\omega(v_*)\right)dv dv_*=0.
\end{align*}
Then, for all $\varphi\in\CC^1_c(A,\R)$, 
\begin{align*}
\int_{U}  \left(g(v)\nabla_v\varphi(v,v_*)-g(v_*)\nabla_{v_*}\varphi(v,v_*)\right)\times\left(\nabla\omega(v)-\nabla\omega(v_*)\right)dv dv_*=0.
\end{align*}
\end{lemma}
\begin{proof}
The proof follows from a standard partition of unity argument (see for instance Th~1.4.4 in~\cite{Hor03}).
\end{proof}
\medskip

We now provide the
\medskip 

\noindent\textit{End of the proof of Proposition \ref{prop:weak_form}.} Consider an open set $U\subset \left(\Rd\right)^2$
and a function $\varphi\in\CC^1_c(U,\R)$. For $\varepsilon\geq 0$, we define
\begin{equation*}
U_{A^c}^\varepsilon = \left\{ (v,v_*)\in U,\ \dist\left((v,v_*),A^c\right)\leq \varepsilon \right\}.
\end{equation*} 
We then consider a mollifier $\rho$ with support included in the unit ball (of $\R^{2d}$), the mollifying sequence of functions $\rho_n$ defined as $\rho_n(v,v_*)=n^{2d}\rho(nv, nv_*)$, and for all $n\in\N^*$ large enough, the sequence of functions $\varphi_n\in\CC^1_c(U,\R)$ defined as
\begin{equation*}
\varphi_n = \left(\rho_{2n}\ast \mathds{1}_{U_{A^c}^\frac{3}{2n}}\right)\varphi .
\end{equation*}
The function $\varphi_n$ satisfies the following properties:
\begin{equation}
\label{hyp:phin_0}
\varphi_n(v,v_*)=0\ \ \forall~(v,v_*)\in U\setminus U_{A^c}^\frac{2}{n},
\end{equation}
\begin{equation}
\label{hyp:phin_phi}
\varphi_n(v,v_*)=\varphi(v,v_*)\ \ \forall~(v,v_*)\in U_{A^c}^\frac{1}{n},
\end{equation}
\begin{equation}
\label{hyp:phin_grad}
\Vert\nabla\varphi_n(v,v_*)\Vert \leq Cn\ \ \forall~(v,v_*)\in U,
\end{equation}
where the constant $C$ may depend on $\varphi$ but not on $n$, and
\begin{equation}
\label{hyp:phin_supp}
\supp(\varphi_n) \subset \supp(\varphi).
\end{equation}
We get
\begin{align*}
&\int_{U}  \left(g(v)\nabla_v\varphi(v,v_*)-g(v_*)\nabla_{v_*}\varphi(v,v_*)\right)\times\left(\nabla\omega(v)-\nabla\omega(v_*)\right)dv dv_* \\
&= \int_{U}  \left(g(v)\nabla_v\left(\varphi(v,v_*)-\varphi_n(v,v_*)\right)-g(v_*)\nabla_{v_*}\left(\varphi(v,v_*)-\varphi_n(v,v_*)\right)\right)\times\left(\nabla\omega(v)-\nabla\omega(v_*)\right)dv dv_* \\
&\quad +\int_{U}  \left(g(v)\nabla_v\varphi_n(v,v_*)-g(v_*)\nabla_{v_*}\varphi_n(v,v_*)\right)\times\left(\nabla\omega(v)-\nabla\omega(v_*)\right)dv dv_*.
\end{align*}
Thanks to~\eqref{hyp:phin_phi}, $\varphi-\varphi_n$ is a $\CC^1$ function whose compact support is included in $A$, therefore we infer from Lemma~\ref{lem:weak_form} and Lemma~\ref{lem:partition_unite} that
\begin{align*}
\int_{U}  \left(g(v)\nabla_v\left(\varphi(v,v_*)-\varphi_n(v,v_*)\right)-g(v_*)\nabla_{v_*}\left(\varphi(v,v_*)-\varphi_n(v,v_*)\right)\right)\times\left(\nabla\omega(v)-\nabla\omega(v_*)\right)dv dv_*=0.
\end{align*}
Besides, since $\nabla\omega$ is of class $\CC^1$ on $\Rd$, and since the support of $\varphi_n$ can be included in a compact subset of $\Rd$ independent of $n$ thanks to~\eqref{hyp:phin_supp}, we see that there exists a constant independent of $n$, still denoted by $C$, such that for all $(v,v_*)$ in $U^\frac{2}{n}_{A^c}\cap \supp(\varphi_n)$,
\begin{equation*}
\left\Vert \nabla\omega(v)-\nabla\omega(v_*) \right\Vert \leq \frac{C}{n}.
\end{equation*}
Using~\eqref{hyp:phin_0} and~\eqref{hyp:phin_grad}, we then estimate
\begin{align}
\label{eq:lim_A}
&\left\Vert \int_{U}  \left(g(v)\nabla_v\varphi(v,v_*)-g(v_*)\nabla_{v_*}\varphi(v,v_*)\right)\times\left(\nabla\omega(v)-\nabla\omega(v_*)\right)dv dv_* \right\Vert  \nonumber\\
&=\left\Vert \int_{U}  \left(g(v)\nabla_v\varphi_n(v,v_*)-g(v_*)\nabla_{v_*}\varphi_n(v,v_*)\right)\times\left(\nabla\omega(v)-\nabla\omega(v_*)\right)dv dv_* \right\Vert  \nonumber\\
&\leq  C\int_{U^\frac{2}{n}_{A^c}\cap\, \supp(\varphi)}  \left(\vert g(v)\vert +\vert g(v_*)\vert\right) n\left\Vert \nabla\omega(v)-\nabla\omega(v_*)\right\Vert dv dv_*   \nonumber\\
&\leq C \int_{\left(U^\frac{2}{n}_{A^c}\setminus A^c\right)\cap\, \supp(\varphi)}  \left(\vert g(v)\vert +\vert g(v_*)\vert\right) dv dv_* .
\end{align}
Because the boundary of $A$ is of Lebesgue measure zero and $g\in L^1_{loc}(\R^d,\R)$, one can use Lebesgue's dominated convergence theorem 
 when $n$ goes to infinity. We obtain
\begin{equation*}
 \int_{U}  \left(g(v)\nabla_v\varphi(v,v_*)-g(v_*)\nabla_{v_*}\varphi(v,v_*)\right)\times\left(\nabla\omega(v)-\nabla\omega(v_*)\right)dv dv_* = 0,
\end{equation*}
and this ends the proof of Proposition~\ref{prop:weak_form}.
\hfill \qed
\bigskip

Before using it to conclude the proof of Theorem~\ref{th:22version2}, we need two additional lemmas.

\begin{lemma}
\label{lem:indep}
Let $p,q\in\CC^1(\Rd,\R)$ be such that the family $(1,p,q)$ is linearly independent in $\CC^1(\Rd,\R)$. If there exist $a,b,c\in\R$ such that for all $v,v_*\in\Rd$,
\begin{equation}
\label{hyp_fq}
a\left(p(v)-p(v_*)\right)^2+b\left(p(v)-p(v_*)\right)\left(q(v)-q(v_*)\right)+c\left(q(v)-q(v_*)\right)^2=0,
\end{equation}
then $a=b=c=0$.
\end{lemma}
\begin{proof}
We consider $i,j\in\{1,\ldots,d\}$ and apply the partial derivatives $\frac{\pa}{\pa_{v_i}}$ and $\frac{\pa}{\pa_{v_{*j}}}$ to~\eqref{hyp_fq}, which yields
\begin{equation*}
\left(2a\pa_jp(v_*)+b\pa_jq(v_*)\right)\pa_ip(v) + \left(2c\pa_jq(v_*)+b\pa_jp(v_*)\right)\pa_iq(v).
\end{equation*}
Hence, for all $v_*\in\Rd$ and all $j\in\{1,\ldots,d\}$, the function
\begin{equation*}
v\mapsto \left(2a\pa_jp(v_*)+b\pa_jq(v_*)\right)p(v) + \left(2c\pa_jq(v_*)+b\pa_jp(v_*)\right)q(v)
\end{equation*}
is constant on $\Rd$, and since $(1,p,q)$ is linearly independent, we must have, for all $v_*\in\Rd$,
\begin{equation*}
2a\pa_jp(v_*)+b\pa_jq(v_*)=0 \quad\text{and}\quad 2c\pa_jq(v_*)+b\pa_jp(v_*)=0.
\end{equation*}
The same argument then yields that $a$, $b$ and $c$ must be zero.
\end{proof}

\begin{lemma}
\label{lem:indep_compact}
Let $p,q\in\CC^0(\Rd,\R)$ be such that the family $(1,p,q)$ is linearly independent in $\CC^0(\Rd,\R)$. Then, there exists a compact subset $K\subset \Rd$ such that $(1,p,q)$ is linearly independent on $K$.
\end{lemma}
\begin{proof}
We argue by contradiction, and assume that there exists an increasing sequence of compacts subsets $K_n\subset \Rd$, such that $\cup_{n\in\N} K_n= \Rd$, and such that $(1,p,q)$ is linearly dependent on each $K_n$. Hence, for all $n\in\N$, there exist $(a_n,b_n,c_n)\in\R^3\setminus\{0\}$, such that for all $n\in\N$,
\begin{equation*}
\theta_n = a_n + b_np + c_nq = 0 \quad \text{on } K_n.
\end{equation*}
Without loss of generality, we can assume that $\left\Vert (a_n,b_n,c_n)\right\Vert=1$ for all $n$. Therefore, up to considering a subsequence, we see that $(a_n,b_n,c_n)_n$ converges towards some $(a_\infty,b_\infty,c_\infty)\neq 0$. This yields that the sequence of functions $(\theta_n)_n$ converges, uniformly on every compact of $\Rd$, towards $\theta_\infty=a_\infty + b_\infty p + c_\infty q$. Besides, on every compact  of $\Rd$, $\theta_n=0$ for $n$ large enough. Thus $\theta_\infty=0$ on $\Rd$, and $(a_\infty,b_\infty,c_\infty)\neq 0$ implies that $(1,p,q)$ is linearly dependent in $\Rd$.
\end{proof}

We are now ready to conclude the proof of the main theorem.
\medskip

\noindent\textit{End of the proof of Theorem~\ref{th:22version2}.}
Let $\alpha,\beta\in\CC^1_c(\Rd,\R)$. Thanks to  Proposition~\ref{prop:weak_form}, with $\varphi(v,v_*)=\alpha(v)\beta(v_*)$, we see that
\begin{align*}
\int_{\R^d}\int_{\R^d}  \left(g(v)\nabla\alpha(v)\beta(v_*)-g(v_*)\alpha(v)\nabla\beta(v_*)\right)\times\left(\nabla\omega(v)-\nabla\omega(v_*)\right)dv dv_* = 0,
\end{align*}
that is, for all $i,j\in\{1,\ldots,d\}$, $i\neq j$,
\begin{align*}
&\int_{\R^d}\int_{\R^d}  \left(\left(g(v)\pa_i\alpha(v)\beta(v_*)-g(v_*)\alpha(v)\pa_i\beta(v_*)\right)\left(\pa_j\omega(v)-\pa_j\omega(v_*)\right) \right. \\
&\hspace{2cm}\left. -\left(g(v)\pa_j\alpha(v)\beta(v_*)-g(v_*)\alpha(v)\pa_j\beta(v_*)\right)\left(\pa_i\omega(v)-\pa_i\omega(v_*)\right) \right)dv dv_*  =0.
\end{align*}
Regrouping the terms differently, we get
\begin{align}
\label{eq:int_1}
&\int_{\R^d} \beta(v_*)dv_* \int_{\R^d} g(v)\left(\pa_i\alpha(v)\pa_j\omega(v)-\pa_j\alpha(v)\pa_i\omega(v)\right)dv \nonumber\\
&+\int_{\R^d} \beta(v_*)\pa_i\omega(v_*)dv_* \int_{\R^d} g(v)\pa_j\alpha(v)dv - \int_{\R^d} \beta(v_*)\pa_j\omega(v_*)dv_* \int_{\R^d} g(v)\pa_i\alpha(v)dv \nonumber\\
&-\int_{\R^d} g(v_*)\pa_i\beta(v_*)dv_* \int_{\R^d} \alpha(v)\pa_j\omega(v)dv + \int_{\R^d} g(v_*)\pa_j\beta(v_*)dv_* \int_{\R^d} \alpha(v)\pa_i\omega(v)dv \nonumber\\
&+\int_{\R^d} g(v_*)\left(\pa_i\beta(v_*)\pa_j\omega(v_*)-\pa_j\beta(v_*)\pa_i\omega(v_*)\right)dv_* \int_{\R^d} \alpha(v)dv =0.
\end{align}
Since $\omega\in\CC^2(\Rd,\R)$, the same computation remains valid with $\beta\pa_i\omega$ and $\beta\pa_j\omega$ instead of $\beta$, which gives
\begin{align}
\label{eq:int_i}
&\int_{\R^d} \beta(v_*)\pa_i\omega(v_*)dv_* \int_{\R^d} g(v)\left(\pa_i\alpha(v)\pa_j\omega(v)-\pa_j\alpha(v)\pa_i\omega(v)\right)dv \nonumber\\
&+\int_{\R^d} \beta(v_*)\left(\pa_i\omega(v_*)\right)^2dv_* \int_{\R^d} g(v)\pa_j\alpha(v)dv - \int_{\R^d} \beta(v_*)\pa_i\omega(v_*)\pa_j\omega(v_*)dv_* \int_{\R^d} g(v)\pa_i\alpha(v)dv \nonumber\\
&-\int_{\R^d} g(v_*)\left(\pa^2_{ii}\omega(v_*)\beta(v_*)+\pa_i\omega(v_*)\pa_i\beta(v_*)\right)dv_* \int_{\R^d} \alpha(v)\pa_j\omega(v)dv \nonumber\\
&+ \int_{\R^d} g(v_*)\left(\pa^2_{ij}\omega(v_*)\beta(v_*)+\pa_i\omega(v_*)\pa_j\beta(v_*)\right)dv_* \int_{\R^d} \alpha(v)\pa_i\omega(v)dv \nonumber\\
&+\int_{\R^d} g(v_*)\left(\left(\pa^2_{ii}\omega(v_*)\beta(v_*)+\pa_i\omega(v_*)\pa_i\beta(v_*)\right)\pa_j\omega(v_*)\right. \nonumber\\
&\hspace{2cm} \left.-\left(\pa^2_{ij}\omega(v_*)\beta(v_*)+\pa_i\omega(v_*)\pa_j\beta(v_*)\right)\pa_i\omega(v_*)\right)dv_* \int_{\R^d} \alpha(v)dv =0,
\end{align}
and
\begin{align}
\label{eq:int_j}
&\int_{\R^d} \beta(v_*)\pa_j\omega(v_*)dv_* \int_{\R^d} g(v)\left(\pa_i\alpha(v)\pa_j\omega(v)-\pa_j\alpha(v)\pa_i\omega(v)\right)dv \nonumber \\
&+\int_{\R^d} \beta(v_*)\pa_i\omega(v_*)\pa_j\omega(v_*)dv_* \int_{\R^d} g(v)\pa_j\alpha(v)dv - \int_{\R^d} \beta(v_*)(\pa_j\omega(v_*))^2dv_* \int_{\R^d} g(v)\pa_i\alpha(v)dv \nonumber\\
&-\int_{\R^d} g(v_*)\left(\pa^2_{ij}\omega(v_*)\beta(v_*)+\pa_j\omega(v_*)\pa_i\beta(v_*)\right)dv_* \int_{\R^d} \alpha(v)\pa_j\omega(v)dv \nonumber\\
&+ \int_{\R^d} g(v_*)\left(\pa^2_{jj}\omega(v_*)\beta(v_*)+\pa_j\omega(v_*)\pa_j\beta(v_*)\right)dv_* \int_{\R^d} \alpha(v)\pa_i\omega(v)dv \nonumber\\
&+\int_{\R^d} g(v_*)\left(\left(\pa^2_{ij}\omega(v_*)\beta(v_*)+\pa_j\omega(v_*)\pa_i\beta(v_*)\right)\pa_j\omega(v_*)\right. \nonumber\\
&\hspace{2cm} \left.-\left(\pa^2_{jj}\omega(v_*)\beta(v_*)+\pa_j\omega(v_*)\pa_j\beta(v_*)\right)\pa_i\omega(v_*)\right)dv_* \int_{\R^d} \alpha(v)dv =0.
\end{align}

The identities~\eqref{eq:int_1}-\eqref{eq:int_j} can be seen as a linear system $Mu=v$, where
\renewcommand{\arraystretch}{1.5}
\begin{equation*}
M=\begin{pmatrix}
\int_{\R^d} \beta(v_*)dv_* & \int_{\R^d}\partial_i\omega(v_*)\beta(v_*)dv_* & \int_{\R^d}\partial_j\omega(v_*)\beta(v_*)dv_* \\
\int_{\R^d}\partial_i\omega(v_*)\beta(v_*)dv_* & \int_{\R^d}(\partial_i\omega(v_*))^2\beta(v_*)dv_* & \int_{\R^d}\partial_i\omega(v_*)\partial_j\omega(v_*)\beta(v_*)dv_* \\
\int_{\R^d}\partial_j\omega(v_*)\beta(v_*)dv_* & \int_{\R^d}\partial_i\omega(v_*)\partial_j\omega(v_*)\beta(v_*)dv_* & \int_{\R^d}(\partial_j\omega(v_*))^2\beta(v_*)dv_* 
\end{pmatrix},
\end{equation*}
\begin{equation*}
u=\begin{pmatrix}
\int_{\R^d} g(v)\left(\pa_i\alpha(v)\pa_j\omega(v)-\pa_j\alpha(v)\pa_i\omega(v)\right)dv \\
\int_{\R^d} g(v)\pa_j\alpha(v)dv \\
-\int_{\R^d} g(v)\pa_i\alpha(v)dv
\end{pmatrix},
\end{equation*}
and 
\begin{equation*}
v=\begin{pmatrix}
\int_{\R^d}\left(v_1+v_2\partial_i\omega(v)+v_3\partial_j\omega(v)\right)\alpha(v)dv\\
\int_{\R^d}\left(v_4+v_5\partial_i\omega(v)+v_6\partial_j\omega(v)\right)\alpha(v)dv\\
\int_{\R^d}\left(v_7+v_8\partial_i\omega(v)+v_9\partial_j\omega(v)\right)\alpha(v)dv
\end{pmatrix}.
\end{equation*}
Explicit expressions of the constants $(v_k)_{1\leq k\leq 9}$ could be obtained from~\eqref{eq:int_1}-\eqref{eq:int_j}, but they are not needed here.

We now consider $i\neq j$, so that $\{1,\pa_i\omega,\pa_j\omega\}$ are linearly independent on $\Rd$ by assumption~\eqref{hyp:linear_indep_rep}. Thanks to Lemma~\ref{lem:indep_compact}, there exists a compact subset $K\subset\Rd$ on which they are linearly independent. Taking for $\beta$ a nonnegative function whose support contains $K$, we get that $M$ is an invertible Gram matrix. Cramer's rule then yields the existence of constants $\left(c_k\right)_{1\leq k\leq 9}$ (again we could get formulas for them, but their explicit expressions will not be needed here) such that, for all $\alpha\in\CC^1_c(\Rd,\R)$,
\begin{equation*}
\label{eq:sol_lin_1}
\int_{\R^d} g(v)\left(\pa_i\alpha(v)\pa_j\omega(v)-\pa_j\alpha(v)\pa_i\omega(v)\right)dv = -\int_{\R^d}\left(c_1+c_2\partial_i\omega(v)+c_3\partial_j\omega(v)\right)\alpha(v)dv,
\end{equation*}
\begin{equation}
\label{eq:sol_lin_2}
\int_{\R^d} g(v)\pa_i\alpha(v)dv = -\int_{\R^d}\left(c_4+c_5\partial_i\omega(v)+c_6\partial_j\omega(v)\right)\alpha(v)dv,
\end{equation}
\begin{equation}
\label{eq:sol_lin_3}
\int_{\R^d} g(v)\pa_j\alpha(v)dv = -\int_{\R^d}\left(c_7+c_8\partial_i\omega(v)+c_9\partial_j\omega(v)\right)\alpha(v)dv.
\end{equation}
Using \eqref{eq:sol_lin_2}-\eqref{eq:sol_lin_3}, we see that $g$ is actually $\CC^1$ on $\Rd$ and that, for all $v\in\Rd$,
\begin{equation}
\label{eq:pa_g}
\pa_i g(v) = c_4+c_5\partial_i\omega(v)+c_6\partial_j\omega(v) \qquad \text{and} \qquad \pa_j g(v) = c_7+c_8\partial_i\omega(v)+c_9\partial_j\omega(v).
\end{equation}
Since $g$ is $\CC^1$,
\begin{equation*}
\left(\nabla g(v) -\nabla g(v_*)\right)\times  \left(\nabla \omega(v) -\nabla \omega(v_*)\right) = 0
\end{equation*}
holds in a strong sense on $\left(\Rd\right)^2$, which yields, for all $v,v_*\in\Rd$
\begin{equation}
\label{eq:prod_vect_coord}
\left(\pa_ig(v)-\pa_ig(v_*)\right)\left(\pa_j\omega(v)-\pa_j\omega(v_*)\right)-\left(\pa_jg(v)-\pa_jg(v_*)\right)\left(\pa_i\omega(v)-\pa_i\omega(v_*)\right)=0.
\end{equation}
Plugging~\eqref{eq:pa_g} into~\eqref{eq:prod_vect_coord}, we obtain, still for all $v,v_*\in\Rd$,
\begin{equation*}
-c_8\left(\pa_i\omega(v)-\pa_i\omega(v_*)\right)^2 + (c_5-c_9)\left(\pa_i\omega(v)-\pa_i\omega(v_*)\right)\left(\pa_j\omega(v)-\pa_j\omega(v_*)\right) + c_6\left(\pa_j\omega(v)-\pa_j\omega(v_*)\right)^2=0.
\end{equation*}
Using once more the independence assumption~\eqref{hyp:linear_indep_rep}, we can then apply Lemma~\ref{lem:indep} and get
\begin{equation*}
c_6=c_8=0 \qquad\text{and}\qquad c_5=c_9.
\end{equation*}
Going back to~\eqref{eq:pa_g}, we get for all $v\in \Rd$:
\begin{equation}
\label{eq:pa_g_simplified}
\pa_i g(v) = c_4+c_5\partial_i\omega(v) \qquad \text{and} \qquad \pa_j g(v) = c_7+c_5\partial_j\omega(v),
\end{equation}
which, if $d=2$, allows us to conclude that there exists constants $b=(c_4,c_7)\in\R^d$ and $c=c_5\in\R$ such that, for all $v\in\Rd$,
\begin{equation*}
\nabla g(v)=b+c\nabla\omega(v).
\end{equation*}
If $d=3$, we assume without loss of generality that $i=1$ and $j=2$. We must have that either $\{1,\pa_1\omega,\pa_3\omega\}$ or $\{1,\pa_2\omega,\pa_3\omega\}$ are linearly independent on $\Rd$. Indeed, if both families were linearly dependent, using that $\{1,\pa_1\omega,\pa_2\omega\}$ are linearly independent, we would get
\begin{equation*}
\kappa_1 +\kappa_2\pa_1\omega = \pa_3\omega = \kappa_3 +\kappa_4\pa_2\omega,
\end{equation*}
for some $\kappa_1,\ldots,\kappa_4$ not all equal to zero, but this would then yield that $\{1,\pa_1\omega,\pa_2\omega\}$ are linearly dependent. Therefore, we can repeat the same argument with either $\{1,\pa_1\omega,\pa_3\omega\}$ or $\{1,\pa_2\omega,\pa_3\omega\}$, and infer also in the case $d=3$ that there exists $b\in\R^d$ and $c\in\R$ such that, for all $v\in\Rd$,
\begin{equation*}
\nabla g(v)=b+c\nabla\omega(v).
\end{equation*}
Hence there also exists $a\in\R$ such that, for all $v\in\Rd$,
\begin{equation*}
g(v)=a+b\cdot v +c\omega(v).
\end{equation*}
\hfill $\qed$

\bigskip 

We now turn to the proof of Theorem \ref{th:22}. This proof is quite short since it uses many elements of the proof of Theorem \ref{th:22version2}.
\medskip

\noindent\textit{Proof of Theorem~\ref{th:22}.} We use the notations of Lemma \ref{lem:para_var}. We consider $(v,v_*) \in A$, and recall equation~\eqref{eq:def_psi}: for all $\sigma$ small enough,
\begin{equation*}
 \omega(v)+\omega(v_*)=\omega(v-\psi(v,v_*,\sigma))+\omega(v_*+\psi(v,v_*,\sigma)).
\end{equation*}
After differentiation of this identity with respect to $\sigma$ at point $0$, we get the formula
$$ (\nabla \omega(v) - \nabla \omega(v_*))\, D_{\sigma}\psi(v,v_*,0) = 0 .$$
It is clear that equation~\eqref{cond:22} is satisfied by $g$. Differentiating now this equation with respect to $\sigma$ at point $0$, we end up with 
$$ (\nabla g(v) - \nabla g(v_*)) D_{\sigma}\psi(v,v_*,0) = 0 . $$
Remembering that (thanks to Lemma \ref{lem:para_var})
\begin{equation*}
\rank(D_\sigma\psi(v,v_*,0))=d-1,
\end{equation*}
we end up with the following formula (which holds for any $(v,v_*) \in A$): 
\begin{equation}\label{stf}
 (\nabla \omega(v) - \nabla \omega(v_*)) \times (\nabla g(v) - \nabla g(v_*)) = 0 . 
 \end{equation}
The same also immediately holds when $(v,v_*) \in A^c$, so that 
(\ref{stf}) holds for all $v,v_*$. This is the strong form of
Proposition \ref{prop:weak_form}, which of course implies the weak formulation, that is Proposition \ref{prop:weak_form} itself. 

The rest of the proof is then identical to that of Theorem \ref{th:22version2}.  
\hfill $\qed$

\section{Proof of Theorem \ref{th:21}} \label{sec3}

We start the proof of Theorem \ref{th:21} by the establishment of the following

\begin{lemma}
\label{lem:diffeo21}
Let $d\geq 2$ and $\omega\in\CC^1(\R^d,\R)$ such that $\omega(0)=0$, $\nabla\omega(0)=0$ and $\nabla\omega(v)\neq 0$ for all $v\neq 0$. For every $\bx\in\Rd$, there exists a bounded neighborhood $U\subset \Rd$, a neighborhood $V\subset\R^{d-1}$ of $0$ and a function $\psi=\psi(v,\sigma)\in\CC^1(\R^d\times\R^{d-1},\R^d)$ such that, for all $v$ in $U$ and all $\sigma$ in $V$,
\begin{equation*}
\psi(v,0)=0 \quad \text{and}\quad \omega(v)+\omega(\psi(v,\sigma))=\omega(v+\psi(v,\sigma)),
\end{equation*}
and such that, for all $v$ in $U$,
\begin{equation*}
\rank\left(D_\sigma\psi(v,0)\right)=d-1.
\end{equation*}
\end{lemma}
\begin{proof}
The proof is similar to the one of Lemma~\ref{lem:para_var}. We consider $\Psi(v,z):=\omega(v+z)-\omega(v)-\omega(z)$, and notice that there exists $i\in\{1,\ldots,d\}$ such that $\partial_i\omega(\bx)\neq 0$, together with $\Psi(\bx,0)=0$, so that we can apply the implicit function theorem and get the existence of a $\CC^1$ function $h$ satisfying
\begin{equation*}
\Psi(v,z)=0 \quad \Leftrightarrow \quad z_i=h(v,z_{i'}),
\end{equation*}
in a neighborhood of $(\bx,0)$, where $z_{i'}=(z_{1},\ldots,z_{i-1},z_{i+1},\ldots,z_{d})$. We then conclude by considering the function $\psi$ defined component-wise by
\begin{equation*}
\psi_j(v,\sigma)=
\left\{\begin{aligned}
&\sigma_j \quad &j<i ,\\
&h(v,\sigma) \quad &j=i , \\
&\sigma_{j-1} \quad &j>i.
\end{aligned}\right.
\end{equation*}
\end{proof}

The following result is a significant intermediate step in the proof of Theorem \ref{th:21}:

\begin{proposition}
\label{prop:mu}
Under the assumptions of Theorem~\ref{th:21}, consider $\tilde g:v\mapsto g(v)-\nabla g(0)\cdot v$. There exists a function $\mu\in\CC(\omega(\R),\R)\cap\CC^1(\omega(\R)\setminus\{0\},\R)$, such that for all $v\in\R^d$,
\begin{equation*}
\tilde g(v)=\mu(\omega(v)).
\end{equation*}
\end{proposition}
\begin{proof}
We first notice that assumption~\eqref{cond:21} is equivalent to having
\begin{equation*}
\omega(v)+\omega(z)=\omega(v+z) \quad\Rightarrow\quad  \tilde g(v)+\tilde g(z)=\tilde g(v+z),
\end{equation*}
for all $(v,z)\in \left(\R^d\right)^2$, the advantage being that we now have $\nabla\tilde g(0)=0$. Using Lemma~\ref{lem:diffeo21}, for all $v\in\Rd$, we obtain
\begin{equation*}
\omega(v)+\omega(\psi(v,\sigma))=\omega(v+\psi(v,\sigma)) \quad \text{and} \quad  \tilde g(v)+ \tilde g(\psi(v,\sigma))= \tilde g(v+\psi(v,\sigma)),
\end{equation*}
for all $\sigma$ in a neighborhood of $0$. By differentiating with respect to $\sigma$ and evaluating at $\sigma=0$, we obtain
\begin{equation*}
D\omega(v)D_\sigma\psi(v,0) = 0 \quad \text{and}\quad D\tilde g(v)D_\sigma\psi(v,0).
\end{equation*}
Since $\rank(D_\sigma\psi(v,0))=d-1$, we infer the existence of $\lambda_v\in\R$ such that
\begin{equation}
\label{eq:grad_col21}
D\tilde g(v) = \lambda_xD\omega(v).
\end{equation}
This yields that, for all $v\in\Rd$, $\tilde g$ is locally constant, and thus globally constant by connectedness, on $\omega^{-1}(\{\omega(v)\})$. Notice that we assumed that $\omega(v)\neq 0$ for all $v\neq 0$, so that
 $\omega^{-1}(\{\omega(0)\})={0}$, and $\tilde g$ is trivially constant on $\omega^{-1}(\{\omega(0)\})$. Hence, for all $a\in\omega(\R)$, $\tilde g$ is constant on $\omega^{-1}(\{a\})$, and we can define the function
\begin{equation*}
\mu : 
\left\{\begin{aligned}
 &\omega(\R) \to \R \\
 &a \mapsto  g\left(\omega^{(-1)}(\{a\})\right).
\end{aligned}\right.
\end{equation*} 

We now prove that $\mu$ indeed belongs to $\CC(\omega(\R),\R)\cap\CC^1(\omega(\R)\setminus\{0\},\R)$. Let $a\in\omega(\R)\setminus\{0\}$ and $v\in\R^d$ such that $\omega(v)=a$. Since $a\neq 0$, $v\neq 0$ and therefore $\nabla\omega(v)\neq 0$,  there exists $i\in\{1,\ldots,d\}$ such that $\partial_i\omega(v)\neq 0$. This yields that the function
\begin{equation*}
\gamma : 
\left\{\begin{aligned}
 &\R\to\R \\
 &t\mapsto \omega(v+te_i),
\end{aligned}\right.
\end{equation*} 
with $e_i$ the $i$-th canonical basis vector of $\R^d$, is a $\CC^1$ diffeomorphism in a neighborhood of $0$. In particular, for $h\in \R$ small enough, there exists a unique $t\in\R$ small enough such that $\gamma(0)+h=\gamma(t)$, i.e. $a+h=\omega(v+te_i)$. We get that $h\to 0 \Leftrightarrow t\to 0$, and more precisely
\begin{equation*}
h=\partial_i\omega(v)t+ o(t),
\end{equation*}
from which we infer
\begin{equation*}
t=(\partial_i\omega(v))^{-1}h + o(h).
\end{equation*}
Therefore, we get
\begin{align*}
\mu(a+h)-\mu(a)&=  g(v+te_i)- g(v) \\
&= \partial_i g(v) t + o(t) \\
&= \lambda_v h + o(h),
\end{align*}
which yields that $\lambda_v$ does not depend on $v$ for $v\in\omega^{-1}(\{a\})$, and that $\mu$ is differentiable at $a$. Besides, since both $ g$ and $\omega$ are assumed to be $\CC^1$, so is $v\mapsto\lambda_v$ thanks to~\eqref{eq:grad_col21}, and therefore $\mu$ is $\CC^1$ at point $a$. We have thus proven that $\mu\in\CC^1(\omega(\R)\setminus\{0\},\R)$. To show that $\mu$ is continuous at point $0$, first notice that assumption~\eqref{cond:21} with $v=z=0$ imposes that $\tilde g(0)=0$,
and thus $\mu(0)=0$. Next, for any $\varepsilon>0$, by continuity of $\tilde g$, there exists $\eta>0$ such that $\Vert v\Vert <\eta \Rightarrow \vert \tilde g(v)\vert <\varepsilon$. With $\delta :=\sup_{\Vert v\Vert <\eta} \omega(v)$, we see that, for all $a\in[0,\delta)$, there exists $v$ such that $\Vert v\Vert <\eta$ and $\omega(v)=a$. From $\tilde g(v)=\mu(\omega(v))$, we infer that $\mu(a)<\varepsilon$, therefore $\mu$ is indeed continuous at $0$.
\end{proof}

We are now ready to conclude the proof of Theorem~\ref{th:21}, by showing that the function $\mu$ is in fact linear.
\medskip

\noindent \textit{End of the proof of Theorem~\ref{th:21}.} 
Let $\ba\in\omega(\R)\setminus\{0\}$ and $\bx\in\Rd$ such that $\omega(\bx)=\ba$. Using Lemma~\ref{lem:diffeo21}, we get the existence of $\varepsilon_1>0$ such that, for all $v\in\R^d$ satisfying $\Vert v-\bx \Vert <\varepsilon_1$ and all $\sigma\in\R^{d-1}$ satisfying $\Vert \sigma \Vert<\varepsilon_1$,
\begin{equation*}
\omega(v)+\omega(\psi(v,\sigma))=\omega(v+\psi(v,\sigma)) \quad \text{and} \quad \tilde g(v)+\tilde g(\psi(\sigma))=\tilde g(v+\psi(v,\sigma)),
\end{equation*}
which yields, using Proposition~\ref{prop:mu},
\begin{equation*}
\mu(\omega(v))+\mu(\omega(\psi(v,\sigma)))=\mu(\omega(v)+\omega(\psi(v,\sigma))).
\end{equation*}
Since $\nabla\omega(\bx)\neq 0$, up to taking $\varepsilon_1$ small enough, there exists $\varepsilon_2>0$ such that, for all $a\in\omega(\R)$ satisfying $\vert a-\ba\vert<\varepsilon_2$, there exists $v\in\R^d$ satisfying $\Vert v-\bx\Vert< \varepsilon_1$ and $\omega(v)=a$.

That is, for all $a\in\omega(\R)$ satisfying $\vert a-\ba\vert<\varepsilon_2$, there exists $v\in\R^d$ satisfying $\Vert v-\bx\Vert< \varepsilon_1$, such that, for all $\sigma\in\R^{d-1}$ satisfying $\Vert \sigma \Vert<\varepsilon_1$,
\begin{equation*}
\mu(a)+\mu(\omega(\psi(v,\sigma)))=\mu(a+\omega(\psi(v,\sigma))).
\end{equation*}
Then, notice that for all $\sigma\neq 0$, we have $\psi(v,\sigma)\neq 0$, and therefore $\omega(\psi(v,\sigma)) \neq 0$, which yields by compactness that
\begin{equation*}
\zeta=\inf_{\Vert v-\bx\Vert< \varepsilon_1} \sup_{\Vert \sigma \Vert<\varepsilon_1} \omega(\psi(v,\sigma)) >0.
\end{equation*}
Hence, we can assume $0<\varepsilon_2<\zeta$ and obtain that for all $v\in\R^d$ satisfying $\Vert v-\bx\Vert< \varepsilon_1$, and all $h\in[0,\varepsilon_2)$, there exists $\sigma\in\R^{d-1}$ such that $\Vert \sigma \Vert<\varepsilon_1$ and $\omega(\psi(v,\sigma))=h$.

Therefore we see that for all $a\in\omega(\R)$ such that $\vert a-\ba\vert<\varepsilon_2$ and all $h\in[0,\varepsilon_2)$,
\begin{equation*}
\mu(a)+\mu(h)=\mu(a+h).
\end{equation*}
%
%
We then infer that, for all all $h\in[0,\varepsilon_2)$,
\begin{equation*}
\mu'(\ba)=\mu'(\ba+h).
\end{equation*}
Hence, $\mu'$ is constant on $\omega(\R)\setminus\{0\}$. Since $\mu(0)=0$, there exists $c\in \R$ such that, for all $a\in\omega(\R)$, $\mu(a)=c\,a$. Therefore $\tilde g(v)=c\,\omega(v)$ for all $v\in\R^d$, and Theorem \ref{th:21} is proven. \hfill $\qed$

\paragraph*{Acknowledgments:} The research leading to this paper was partly funded by Universit\'e Sorbonne Paris Cit\'e, in the framework of the ``Investissements d'Avenir'', 
convention ANR-11-IDEX-0005. MB also acknowledges partial support from a Lichtenberg Professorship grant of the VolkswagenStiftung awarded to C. Kuehn.


\begin{thebibliography}{99}


\bibitem{Bal91}
A.~Balk, ``A new invariant for rossby wave systems,'' {\em Physics Letters A},
  vol.~155, no.~1, pp.~20--24, 1991.

\bibitem{BalFer98}
A.~Balk and E.~Ferapontov, ``Invariants of wave systems and web geometry,''
  {\em American Mathematical Society Translations}, pp.~1--30, 1998.

\bibitem{baranger_mouhot}
C. Baranger, C. Mouhot,
\newblock 
Explicit spectral gap estimates for the linearized Boltzmann and Landau operators with hard potentials.
\newblock {\em Rev. Mat. Iberoamericana}, {\bf 21}
, n.3, (2005), 819--841.

\bibitem{cercignani}
C. Cercignani, 
\newblock
Theory and application of the Boltzmann equation. Elsevier, New York, 1975.

\bibitem{cercignani_rel}
C. Cercignani, G. M. Kremer
\newblock 
The relativistic Boltzmann equation. Theory and Applications. Springer, 2002.

\bibitem{LD89}
L.~Desvillettes.
\newblock Entropy dissipation rate and convergence in kinetic equations.
\newblock {\em Comm. Math. Phys.}, {\bf 123} n.4, (1988), {687--702}.

\bibitem{landau_vill2}
L.~Desvillettes and C.~Villani.
\newblock On the spatially homogeneous {L}andau equation for hard potentials.
  {P}art {II}. H-Theorem and applications. 
\newblock {\em Commun. Partial Differential Equations},  {\bf{25}}, n.1-2 (2000), 261-298.

\bibitem{landau_jfa}
L.~Desvillettes.
\newblock
Entropy dissipation estimates for the Landau equation in the Coulomb case and applications.
\newblock {\em J. Functional Anal.}, {\bf{269}}, (2015), 1359-1403.


\bibitem{landau_braga}
L.~Desvillettes.
\newblock
Entropy dissipation estimates for the Landau equation: General cross sections. Proceedings of the congress PSPDE III, Braga, 2014. 

\bibitem{survey}
L. Desvillettes, C. Mouhot and C. Villani.
\newblock
Celebrating Cercignani's conjecture for the Boltzmann equation.
\newblock {\em Kinetic and Related Models}, {\bf{4}}, n.1, (2011),  277--294.


\bibitem{germ}
P. Germain, A. D. Ionescu and M.-B. Tran.
\newblock
Optimal local well-posedness theory for the kinetic wave equation.
\newblock {\em Preprint}, arXiv:1711.05587.


\bibitem{EscMisVal03}
M.~Escobedo, S.~Mischler, and M.~A. Valle, {\em Homogeneous Boltzmann equation
  in quantum relativistic kinetic theory}.
\newblock Department of Mathematics, Texas State University-San Marcos, 2003.


\bibitem{Hor03}
L.~H{\"o}rmander, {\em The analysis of linear partial differential operators.
  I. Distribution theory and Fourier analysis. Reprint of the second (1990)
  edition}.
\newblock Springer, Berlin, 2003.


\bibitem{strain}
R.~M. Strain, and M. Tasković, {\em Entropy dissipation estimates for the relativistic Landau equation, and applications}.
\newblock Preprint, arXiv:1806.08720.

\bibitem{wennberg}
B. Wennberg {\em On an entropy dissipation inequality for the Boltzmann equation}.
\newblock Comptes-Rendus Acad. Sc., {\bf{315}},  (1992),  1441--1446.

\bibitem{ZakLvoFal12}
V.~E. Zakharov, V.~S. L'vov, and G.~Falkovich, {\em Kolmogorov spectra of
  turbulence I: Wave turbulence}.
\newblock Springer Science \& Business Media, 2012.

\end{thebibliography}
\end{document}